\documentclass[12pt]{article}

\usepackage{amssymb}
\usepackage{amsthm}
\usepackage{amsmath}
\usepackage{graphicx}
\usepackage{fullpage}
\usepackage{color}
 \numberwithin{equation}{section}

\theoremstyle{plain}
\newtheorem{thm}{Theorem}[section]
\newtheorem{cor}[thm]{Corollary}

\newtheorem{lem}[thm]{Lemma}
\newtheorem{prop}[thm]{Proposition}

\theoremstyle{definition}

\theoremstyle{remark}

\newtheorem{rem}[thm]{Remark}

\newcommand{\N}{\mathbb{N}}
\newcommand{\R}{\mathbb{R}}

\newcommand{\I}{\infty}
\newcommand{\bp}{\begin{proof}[\ensuremath{\mathbf{Proof}}]}
\newcommand{\ep}{\end{proof}}

\begin{document}


\title{Inverse iteration for $p$-ground states}

\author{Ryan Hynd\footnote{Department of Mathematics, University of Pennsylvania.  Partially supported by NSF grant DMS-1301628.}\; and Erik Lindgren\footnote{Supported by the Swedish Research Council, grant no. 2012-3124.}}  

\maketitle

\begin{abstract}
We adapt the inverse iteration method for symmetric matrices to some nonlinear PDE eigenvalue problems. In particular, for $p\in (1,\infty)$ and a given domain $\Omega\subset\R^n$, we analyze
a scheme that allows us to approximate the smallest value the ratio $\int_\Omega|D\psi|^pdx/\int_\Omega|\psi|^pdx$ can assume for functions $\psi$ that vanish on $\partial \Omega$.  The scheme in question 
also provides a natural way to approximate minimizing $\psi$.  Our analysis also extends in the limit as $p\rightarrow\infty$ and thereby fashions a new approximation method for ground states of the infinity Laplacian. 
\end{abstract}

\section{Introduction}
\par In this paper, we will use a generalization of the inverse iteration method for symmetric matrices to estimate solutions of some nonlinear PDE eigenvalue problems.  The first problem we consider is as follows. For $p\in (1,\infty)$ and a bounded domain $\Omega\subset \R^n$, we define
\begin{equation}\label{LamPee}
\lambda_p:=\inf\left\{\frac{\int_{\Omega}|D\psi|^pdx }{\int_{\Omega}|\psi|^pdx }: \psi\in W^{1,p}_0(\Omega),\; \psi\not\equiv 0\right\}.
\end{equation}
Here $W^{1,p}_0(\Omega)$ is the closure of the smooth, compactly supported functions $\phi:\Omega\rightarrow \R$ in the norm $\left(\int_\Omega|D\phi|^pdx\right)^{1/p}$; we refer 
readers to the sources \cite{Evans, Mazya} for information on Sobolev spaces and their applications to PDE .  It is evident that $1/\lambda_p$ is the smallest constant $C$ for which 
the {\it Poincar\'{e} inequality} 
$$
\int_{\Omega}|\psi|^pdx\le C \int_{\Omega}|D\psi|^pdx, \quad\psi\in W^{1,p}_0(\Omega)
$$
holds. 

\par The constant $\lambda_p$ is also a type of eigenvalue.  Indeed, minimizers in \eqref{LamPee} are called {\it $p$-ground states} and satisfy the PDE
$$
\begin{cases}
-\Delta_pu=\lambda_p|u|^{p-2}u, \quad & x\in \Omega,\\
\hspace{.34in} u=0, \quad &x\in \partial \Omega.
\end{cases}
$$
Here, the operator $\Delta_p\psi:=\text{div}(|D\psi|^{p-2}D\psi)$ is known as the $p$-Laplacian.  It has been established that $p$-ground states exist and that any two are 
multiples of one another, see \cite{LindKaw, Saka}.  Consequently, $\lambda_p$ is said to be {\it simple}. 

\par We will use the following iteration scheme to approximate $\lambda_p$ and $p$-ground states. Let $u_0\in L^p(\Omega)$, and consider the family of PDE
\begin{equation}\label{IterationScheme}
\begin{cases}
-\Delta_pu_k=|u_{k-1}|^{p-2}u_{k-1}, \quad & x\in \Omega\\
\hspace{.34in} u_k=0, \quad &x\in \partial \Omega
\end{cases}
\end{equation}
for $k\in \N$.  It can be verified without too much difficulty that for a given $u_0$, there is a unique weak solution sequence $(u_k)_{k\in \N}\subset W^{1,p}_0(\Omega)$ of \eqref{IterationScheme}.  That is, there is a unique sequence $(u_k)_{k\in \N}\subset W^{1,p}_0(\Omega)$ such that 
$$
\int_\Omega |Du_k|^{p-2}Du_k\cdot D\phi dx=\int_\Omega |u_{k-1}|^{p-2}u_{k-1}\phi dx
$$
for each $\phi \in W^{1,p}_0(\Omega)$ and $k\in \N$. In fact, once $u_{k-1}\in L^p(\Omega)$ is known, $u_k$ can be obtained by minimizing the functional 
$$
W^{1,p}_0(\Omega)\ni v\mapsto \int_{\Omega}\left(\frac{1}{p}|Dv|^p - |u_{k-1}|^{p-2}u_{k-1}v\right)dx.
$$ 
As this functional is strictly convex and coercive, the existence of a unique minimizer follows from the ``direct method" of the calculus of variations. 

\par The following theorem details how the scheme \eqref{IterationScheme} is related to $\lambda_p$ and $p$-ground states. 

\begin{thm}\label{Thm1}
Assume $u_0\in L^p(\Omega)$ and define
$$
\mu_p:=\lambda^{\frac{1}{p-1}}_p.
$$
Then the limit 
$$
\psi:=\lim_{k\rightarrow \infty}\mu_p^{k}u_k
$$
exists in $W^{1,p}_0(\Omega)$. If $\psi \not\equiv 0$, then $\psi$ is a $p$-ground state and 
\begin{equation}\label{LamPeeLimit}
\lambda_p=\lim_{k\rightarrow\infty}\frac{\int_{\Omega}|Du_k|^pdx }{\int_{\Omega}|u_k|^pdx }.
\end{equation}
\end{thm}
\begin{rem}\label{nonzero}
It may not be obvious how to verify that the limiting function $\psi$ is not identically zero. However, if for instance $u_0>0$ in $\overline \Omega$ or if $u_0\geq 0$ and $\Omega$ is regular enough in order to have a Hopf's lemma (for instance $C^{1,\alpha}$, cf. \cite{MS}), then it is straightforward to verify that $\psi$ is indeed non-zero.
\end{rem}
The iteration scheme \eqref{IterationScheme} was introduced by R. Biezuner, G. Ercole, and E. Martins in \cite{biezuner} who conjectured the limit
\begin{equation}\label{BEMconj}
\lambda_p=\lim_{k\rightarrow\infty}\left(\frac{\int_\Omega|u_{k-1}|^pdx}{\int_\Omega|u_{k}|^pdx}\right)^{1-1/p}.
\end{equation}
We prove this limit holds under the hypotheses of Theorem \ref{Thm1}; see Corollary \ref{BiezCor}.  We also show that the sequences
$$
\left(\frac{\int_{\Omega}|Du_k|^pdx }{\int_{\Omega}|u_k|^pdx }\right)_{k\in \N}\quad \text{and}\quad 
\left(\frac{\int_\Omega|u_{k-1}|^pdx}{\int_\Omega|u_{k}|^pdx}\right)_{k\in \N}
$$
are nonincreasing, which we regard as special features of the the iteration \eqref{IterationScheme}. See Proposition \ref{PropProp} below.  

\par Next, we derive an iteration scheme in the limit as $p\rightarrow \infty$.  Our motivation was the seminal work of P. Juutinen, P. Lindqvist, and J. Manfredi \cite{JLM}, where it was proven 
that $\lim_{p\rightarrow\infty}\lambda_p^{1/p}$ exists and equals 
$$
\lambda_\infty:=\inf\left\{\frac{|D\psi|_{L^\infty(\Omega)}}{|\psi|_{L^\infty(\Omega)}}: \psi\in W^{1,\infty}_0(\Omega),\; \psi\not\equiv 0\right\}=\left(\sup \{r: B_r(x)\subset \Omega \text{ for some $x\in \Omega$}\}\right)^{-1}.
$$
Here $W^{1,\infty}_0(\Omega)$ is the space of Lipschitz continuous functions $\psi:\overline{\Omega}\rightarrow\R$ that satisfy $\psi|_{\partial\Omega}=0$. 
Furthermore, these authors also showed that there is a sequence $(u_{p_j})_{j\in \N}$ of $p$-ground states that converge uniformly to a viscosity solution $w\in W^{1,\infty}_0(\Omega)$ of the PDE
\begin{equation}\label{InfinityGroundStates}
0=
\begin{cases}
\min\{-\Delta_\infty w, |Dw| - \lambda_\infty w\}, \quad &w>0,\\
-\Delta_\infty w, \quad &w=0,\\
\max\{-\Delta_\infty w, -|Dw| - \lambda_\infty w\}, \quad &w<0.\\
\end{cases}
\end{equation}
Here $\Delta_\infty \psi:=D^2\psi D\psi\cdot D\psi$ is the infinity Laplacian and 
nontrivial solutions of \eqref{InfinityGroundStates} having constant sign, are called {\it $\infty$-ground states}.

\par Passing to the limit as $p\rightarrow \infty$ in \eqref{IterationScheme}, we are able to conclude the subsequent result.  The novelty in the 
theorem below is that \eqref{InfinityIterationScheme} presents a new mechanism for generating $\infty$-ground states.

\begin{thm}\label{Thm2}
Assume $u_0\in C(\overline{\Omega})$ and denote $(u_{k,p})_{k\in \N}$ as the solution sequence of \eqref{IterationScheme}. \\
(i) There is a sequence $(p_j)_{j\in \N}$ increasing to $\infty$ and $(v_k)_{k\in \N}\subset W^{1,\infty}_0(\Omega)$ such that $u_{k,p_j}\rightarrow v_k$ 
uniformly on $\overline{\Omega}$ as $j\rightarrow\infty$ for each $k\in \N$. Moreover, $v_k$ is a viscosity solution of the PDE
\begin{equation}\label{InfinityIterationScheme}
0=
\begin{cases}
\min\{-\Delta_\infty v_k, |Dv_k| - v_{k-1}\}, \quad &v_{k-1}>0\\
-\Delta_\infty v_k, \quad &v_{k-1}=0\\
\max\{-\Delta_\infty v_k, -|Dv_k| - v_{k-1}\}, \quad &v_{k-1}<0\\
\end{cases} 
\end{equation}
for each $k\in \N$. (Here $v_0:=u_0$.)\\
(ii) The limit $L:=\lim_{k\rightarrow \infty}\lambda_\infty^k|Dv_k|_{L^\infty(\Omega)}$ exists. If $L>0$, 
$$
\lambda_\infty=\lim_{k\rightarrow\infty}\frac{|Dv_{k}|_{L^\infty(\Omega)}}{|v_{k}|_{L^\infty(\Omega)}}.
$$
and any uniformly convergent subsequence of $(\lambda_\infty^{k}v_{k})_{k\in \N}$ converges to a solution of \eqref{InfinityGroundStates}.
\end{thm}
\begin{rem}
Obviously, if $u_0\geq 0$ and $L>0$, then any uniformly convergent subsequence of $(\lambda_\infty^{k}v_{k})_{k\in \N}$ converges to an $\infty$-ground state.
\end{rem}

\par We would especially like to thank Richard Tapia. After learning about our previous work \cite{HyndLindgren} which employed a doubly nonlinear flow to approximate $\lambda_p$ and 
$p$-ground states, Professor Tapia suggested that it may be possible to use inverse iteration to obtain similar results.  As noted above, the authors R. Biezuner, G. Ercole, and E. Martins were
the first to make this observation in \cite{biezuner}.  Nevertheless, we believe this paper adds significantly to \cite{biezuner} and makes clear the connection between inverse iteration and $p$-ground states.

\section{Convergence of the scheme}
Before proving Theorem \ref{Thm1}, we will first make an observation which illuminates how $\mu_p$ enters the statement of 
the theorem.  In particular, we will argue that $(\mu^k_pu_k)_{k\in \N}$ is bounded in $W^{1,p}_0(\Omega)$ and
$\left(\mu_p^k |Du_k|_{L^p(\Omega)}\right)_{k\in \N}$ is a nonincreasing sequence of real numbers. 
\begin{lem}\label{DecreaseW1p}
For each $k\in \N$, 
$$
\mu_p^p\int_\Omega|Du_{k+1}|^pdx\le \int_\Omega|Du_{k}|^pdx.
$$
\end{lem}
\begin{proof} Assume $\int_\Omega|Du_{k+1}|^pdx\neq 0$. We employ H\"{o}lder's inequality and the Poincar\'{e} inequality to find  
\begin{align}\label{FirstIneq}
\int_{\Omega}|Du_{k+1}|^pdx&=\int_{\Omega}|Du_{k+1}|^{p-2}Du_{k+1}Du_{k+1}dx\nonumber \\
&=\int_{\Omega}|u_{k}|^{p-2}u_{k} u_{k+1}dx\nonumber \\
&\le \left(\int_\Omega|u_{k}|^pdx\right)^{1-1/p}\left(\int_\Omega|u_{k+1}|^pdx\right)^{1/p}\\
&\le \left(\frac{1}{\lambda_p}\int_\Omega|Du_{k}|^pdx\right)^{1-1/p}\left(\frac{1}{\lambda_p}\int_\Omega|Du_{k+1}|^pdx\right)^{1/p}\nonumber \\
&= \frac{1}{\lambda_p}\left(\int_\Omega|Du_{k}|^pdx\right)^{1-1/p}\left(\int_\Omega|Du_{k+1}|^pdx\right)^{1/p}.\nonumber
\end{align}
Consequently, 
$$
\int_{\Omega}|Du_{k+1}|^pdx\le  \frac{1}{\lambda_p^{p/(p-1)}}\int_{\Omega}|Du_{k}|^pdx
$$
which proves the claim. 
\end{proof}
\begin{rem}
A minor variation in the proof of Lemma \ref{DecreaseW1p} gives the estimate 
\begin{equation}\label{ThreeHalfsIneq}
\int_\Omega|Du_{k}|^pdx\le\frac{1}{\mu_p}\int_\Omega|u_{k-1}|^pdx
\end{equation}
for each $k\in \N$. This estimate will be employed in the proof of Theorem \ref{Thm2}.
\end{rem}
\begin{proof}[~Proof of Theorem \ref{Thm1}]
Set $\psi_k:=\mu_p^{k}u_k$ $(k\in \N)$ and 
$$
S:=\lim_{k\rightarrow \infty}\int_{\Omega}|D\psi_k|^pdx.
$$
Observe that the limit defining $S$ exists by Lemma \ref{DecreaseW1p}.  If $S=0$, the assertion follows. Let us now assume otherwise.  

\par Notice that $(\psi_k)_{k\in \N}$ satisfies the sequence of PDE
$$
\begin{cases}
-\Delta_p\psi_k=\lambda_p |\psi_{k-1}|^{p-2}\psi_{k-1}, \quad & x\in \Omega,\\
\hspace{.34in}\psi_k=0, \quad &x\in \partial \Omega.
\end{cases}
$$
By Lemma \ref{DecreaseW1p} and Rellich-Kondrachov compactness, there is $\psi\in W^{1,p}_0(\Omega)$ and a subsequence $(\psi_{k_j})_{j\in \N}$ so that 
$\psi_{k_j}\rightarrow \psi$ in $L^p(\Omega)$ and $D\psi_{k_j}\rightharpoonup D\psi$ in $L^p(\Omega;\R^n)$, as $j\rightarrow\infty$. Also note 
\begin{align*}
\int_\Omega |D\psi_{k_j}|^pdx& =\int_\Omega |D\psi_{k_j}|^{p-2}D\psi_{k_j}\cdot D\psi_{k_j}dx=\lambda_p\int_\Omega |\psi_{k_j-1}|^{p-2}\psi_{k_j-1}\psi_{k_j}dx.
\end{align*} 
\par Since $\psi_{k_j}\rightarrow \psi$ in $L^p(\Omega)$, 
$$
\limsup_{j\rightarrow \infty}\int_\Omega |D\psi_{k_j}|^pdx= \lambda_p \int_\Omega|\psi|^pdx\le \int_\Omega|D\psi|^pdx.
$$
And the weak convergence $D\psi_{k_j}\rightharpoonup D\psi$ in $L^p(\Omega;\R^n)$ gives 
$$
\liminf_{j\rightarrow \infty}\int_\Omega |D\psi_{k_j}|^pdx\ge \int_\Omega|D\psi|^pdx.
$$
Thus, $\psi_{k_j}\rightarrow \psi$ in $W^{1,p}_0(\Omega)$, $S=\int_\Omega |D\psi|^pdx$ and
$$
 \int_\Omega|D\psi|^pdx=\lambda_p\int_\Omega|\psi|^pdx.
$$
\par As $S>0$, $\psi\not\equiv 0$ and thus $\psi$ is a $p$-ground state. Since $S$ is the same for all any subsequential limit, the simplicity of $\lambda_p$ implies that $\psi_k\rightarrow \psi$ in $W^{1,p}_0(\Omega)$ as claimed. Moreover, 
$$
\lim_{k\rightarrow\infty}\frac{\int_\Omega|Du_k|^pdx}{\int_\Omega|u_k|^pdx}=\lim_{k\rightarrow\infty}\frac{\int_\Omega|D\psi_k|^pdx}{\int_\Omega|\psi_k|^pdx}=
\frac{\int_\Omega|D\psi|^pdx}{\int_\Omega|\psi|^pdx}=\lambda_p.
$$
\end{proof}
Observe that if $u_0$ is a $p$-ground state, then $(\mu_p^{-k}u_0)_{k\in \N}$ is a ``separation of variables"  solution of \eqref{IterationScheme}. This is a 
trivial case of Theorem \ref{Thm1}. Also note that $\lim_{k\rightarrow\infty}\mu_p^{k}u_k$ could vanish identically. For instance, this occurs when $p=2$ and $u_0$ is an eigenfunction of the Dirichlet Laplacian corresponding to an eigenvalue different that $\lambda_2$.   Let us now see how the limit \eqref{BEMconj} follows from
Theorem \ref{Thm1}.
 
\begin{cor}\label{BiezCor}
Assume $\lim_{k\rightarrow\infty}\mu_p^{k}|Du_k|_{L^p(\Omega)}\not\equiv 0$, then the limit 
\eqref{BEMconj} holds.
\end{cor}
\begin{proof}
Set $\psi_k:=\mu_p^{k}u_k$. By the previous assertion, $(\psi_k)_{k\in \N}$ converges to a $p$-ground state in $W^{1,p}_0(\Omega)$.    As a result, 
$$
\lim_{k\rightarrow \infty}\frac{\int_\Omega|u_{k-1}|^pdx}{\int_\Omega|u_{k}|^pdx}=
\mu_p^p\lim_{k\rightarrow \infty}\frac{\int_\Omega|\psi_{k-1}|^pdx}{\int_\Omega|\psi_{k}|^pdx}=\lambda^{p/(p-1)}_p.
$$
\end{proof}
We conclude this section by establishing some fundamental properties of the iteration scheme \eqref{IterationScheme}. The monotonicity \eqref{PropUno}
suggests the iteration scheme improves the Rayleigh quotient $\int_\Omega|D\psi|^pdx/\int_\Omega|\psi|^pdx$ at each step, and the monotonicity 
\eqref{PropDos} gives more insight on the limit \eqref{BEMconj}. 
\begin{prop}\label{PropProp}
Assume $u_0\in W^{1,p}_0(\Omega)$ and $u_0\not\equiv 0$. Then $u_k\not\equiv 0$ for each $k\in \N$, 
\begin{equation}\label{PropUno}
\frac{\int_\Omega|Du_{k+1}|^pdx}{\int_\Omega|u_{k+1}|^pdx}\le \frac{\int_\Omega|Du_{k}|^pdx}{\int_\Omega|u_{k}|^pdx},
\end{equation}
and
\begin{equation}\label{PropDos}
\frac{\int_\Omega|u_{k}|^pdx}{\int_\Omega|u_{k+1}|^pdx}\le \frac{\int_\Omega|u_{k-1}|^pdx}{\int_\Omega|u_{k}|^pdx}
\end{equation}
for each $k\in \N$.  
\end{prop}
\begin{proof}
If $u_0\not\equiv 0$, then $u_1\not\equiv 0$ or \eqref{IterationScheme} could not hold when $k=1$. By induction, we may conclude $u_k\not\equiv 0$ for each $k\in \N$.  \par Now fix $k\in \N$ and observe
\begin{align}\label{SecIneq}
\int_\Omega|u_{k}|^pdx&=\int_\Omega(|u_{k}|^{p-2}u_{k})u_{k}dx\nonumber \\
&=\int_\Omega|Du_{k+1}|^{p-2}Du_{k+1}\cdot Du_{k}dx\nonumber \\
&\le \left(\int_\Omega|Du_{k+1}|^pdx\right)^{1-1/p}\left(\int_\Omega|Du_{k}|^pdx\right)^{1/p}.
\end{align}
Combining the bound \eqref{FirstIneq} with \eqref{SecIneq} gives
\begin{align*}
\frac{\int_\Omega|Du_{k+1}|^pdx}{\int_\Omega|u_{k+1}|^pdx}&\le \frac{ \left(\int_\Omega|u_{k}|^pdx\right)^{1-1/p}\left(\int_\Omega|u_{k+1}|^pdx\right)^{1/p}}{\int_\Omega|u_{k+1}|^pdx}\\
&=\frac{ \int_\Omega|u_{k}|^pdx}{\left(\int_\Omega|u_{k+1}|^pdx\right)^{1-1/p} \left(\int_\Omega|u_{k}|^pdx\right)^{1/p}}\\
&\le \frac{\left(\int_\Omega|Du_{k+1}|^pdx\right)^{1-1/p}\left(\int_\Omega|Du_{k}|^pdx\right)^{1/p}}{\left(\int_\Omega|u_{k+1}|^pdx\right)^{1-1/p} \left(\int_\Omega|u_{k}|^pdx\right)^{1/p}}\\
&=\left(\frac{\int_\Omega|Du_{k+1}|^pdx}{\int_\Omega|u_{k+1}|^pdx}\right)^{1-1/p}\left(\frac{\int_\Omega|Du_{k}|^pdx}{\int_\Omega|u_{k}|^pdx}\right)^{1/p},
\end{align*}
which verifies \eqref{PropUno}.

\par As for \eqref{PropDos}, we employ \eqref{SecIneq}, \eqref{PropUno} and \eqref{FirstIneq} to find
\begin{align*}
\frac{\int_\Omega|u_{k}|^pdx}{\int_\Omega|u_{k+1}|^pdx}&\le  \frac{\left(\int_\Omega|Du_{k+1}|^pdx\right)^{1-1/p}\left(\int_\Omega|Du_{k}|^pdx\right)^{1/p}}{\int_\Omega|u_{k+1}|^pdx}\\
&\le\left[\int_\Omega|u_{k+1}|^pdx\;\frac{\int_\Omega|Du_{k}|^pdx}{\int_\Omega|u_{k}|^pdx}\right]^{1-1/p}\frac{\left(\int_\Omega|Du_{k}|^pdx\right)^{1/p}}{\int_\Omega|u_{k+1}|^pdx}\\
&=\frac{\int_\Omega|Du_{k}|^pdx}{\left(\int_\Omega|u_{k+1}|^pdx\right)^{1/p}\left(\int_\Omega|u_{k}|^pdx\right)^{1-1/p}}\\
&\le \frac{\left(\int_\Omega|u_{k}|^pdx\right)^{1/p}\left(\int_\Omega|u_{k-1}|^pdx\right)^{1-1/p}}{\left(\int_\Omega|u_{k+1}|^pdx\right)^{1/p}\left(\int_\Omega|u_{k}|^pdx\right)^{1-1/p}}\\
&=\left(\frac{\int_\Omega|u_{k}|^pdx}{\int_\Omega|u_{k+1}|^pdx}\right)^{1/p}\left(\frac{\int_\Omega|u_{k-1}|^pdx}{\int_\Omega|u_{k}|^pdx}\right)^{1-1/p}.
\end{align*}
\end{proof}
\begin{rem}
If $u_0\not\equiv 0$, the sequences
$$
\left(\frac{\int_{\Omega}|Du_k|^pdx }{\int_{\Omega}|u_k|^pdx }\right)_{k\in \N}\quad \text{and}\quad 
\left(\frac{\int_\Omega|u_{k-1}|^pdx}{\int_\Omega|u_{k}|^pdx}\right)_{k\in \N}
$$
are bounded below by $\lambda_p$ and $\lambda_p^{p/(p-1)}$, respectively; see Proposition 2.8 of \cite{biezuner}. In view of the monotonicity
\eqref{PropUno} and \eqref{PropDos}, both of these sequences are convergent.  However, the limits \eqref{LamPeeLimit} and \eqref{BEMconj} may not 
hold if $\lim_{k\rightarrow\infty}\mu_p^{k}u_k\equiv 0$. For example, these limits fail if $p=2$ and $u_0$ is an eigenfunction of the Dirichlet Laplacian that corresponds to an eigenvalue not equal to $\lambda_2$. 
\end{rem}

\section{The large $p$ limit}
This section is dedicated to a proof of Theorem \ref{Thm2}, which characterizes the large $p$ limit of the solutions of the iteration scheme \eqref{IterationScheme}.  We
begin with an important observation regarding weak solution sequences $(u_k)_{k\in \N}\subset W^{1,p}_0(\Omega)$ of \eqref{IterationScheme}
when $u_0\in C(\overline{\Omega})$
\begin{lem}
Suppose $u_0\in C(\overline{\Omega})$, and let $(u_k)_{k\in \N}\subset W^{1,p}_0(\Omega)$ denote the associated solution sequence of \eqref{IterationScheme}. 
Then for each $k\in \N$, there is $\alpha_k\in (0,1)$ such that 
$$
u_k\in C^{1,\alpha_k}_{\text{loc}}(\Omega)\cap L^\infty(\Omega).
$$
\end{lem}
\begin{proof}
It suffices to verify the claim for $k=1$; the case $k\ge 2$ then follows from induction.  Recall that \eqref{IterationScheme} implies $u_1\in W^{1,p}_0(\Omega)$ is a weak solution of
 solution of 
$$
\begin{cases}
-\Delta_p u_1=|u_0|^{p-2}u_0, \quad & x\in \Omega, \\
\hspace{.34in} u_1=0, \quad &x\in \partial \Omega.
\end{cases}
$$
We will use a weak comparison principle argument to bound $u_1$ from above and then from below. The regularity theory developed by E. DiBenedetto in
\cite{DB} would then imply the existence of an $\alpha_1\in (0,1)$ such that $u_1\in C^{1,\alpha_1}_{\text{loc}}(\Omega)$.

\par To this end, we fix any $y\notin \overline{\Omega}$ and define 
$$
w(x):=\frac{1}{qn^{\frac{1}{p-1}}}|x-y|^q, \quad x\in \overline{\Omega}.
$$
Here $q=p/(p-1)$ is the H\"{o}lder exponent dual to $p$. Direct computation has $\Delta_pw(x)=1$ for each $x\in \Omega$. 
It is also routine to verify that
$$
v:=|u_0|_{L^\infty(\Omega)}\left(|w|_{L^\infty(\Omega)}-w\right)
$$
satisfies 
$$
-\Delta_pv\ge |u_0|^{p-2}u_0, \quad x\in \Omega.
$$
Since $v|_{\partial\Omega}\ge 0=u_1|_{\partial\Omega}$, a standard weak comparison argument implies $u_1\le v$ in $\Omega$. In particular, 
$$
u_1\le |w|_{L^\infty(\Omega)}|u_0|_{L^\infty(\Omega)}, \quad x\in \Omega.
$$
We can argue similarly to bound $u$ from below and derive 
$$
u_1\ge -|w|_{L^\infty(\Omega)}|u_0|_{L^\infty(\Omega)}, \quad x\in \Omega.
$$
\end{proof}
We have just established that the solution sequence $(u_k)_{k\in \N}$ of the inverse iteration scheme is continuous, provided that $u_0$ is continuous. 
Virtually the same argument given by P. Juutinen, P. Lindqvist and J. Manfredi in the proof of Theorem 2.5 of \cite{JLM2} implies that each $u_k$ is additionally a 
viscosity solution of \eqref{IterationScheme}. That is, each solution sequence $(u_k)_{k\in \N}\subset C(\overline{\Omega})$ of \eqref{IterationScheme} with $p\ge 2$ has the following property. For each $k\in \N$,
$$
-\Delta_p\phi(x_0)\le |u_{k-1}(x_0)|^{p-2}u_{k-1}(x_0)
$$
whenever $\phi\in C^2(\Omega)$ and $u_k-\phi$ has a local maximum at $x_0\in \Omega$, and 
$$
-\Delta_p\phi(x_0)\ge |u_{k-1}(x_0)|^{p-2}u_{k-1}(x_0)
$$
whenever $\phi\in C^2(\Omega)$ and $u_k-\phi$ has a local minimum at $x_0\in \Omega$. We refer interested readers to the ``user's guide to viscosity solutions" \cite{CIL} for more information on viscosity solutions of elliptic PDE, and we are now ready to prove Theorem \ref{Thm2}. 
\begin{proof}[~Proof of Theorem \ref{Thm2} part $(i)$]  Employing 
 Lemma \ref{DecreaseW1p} and inequality \eqref{ThreeHalfsIneq} for $k=1$ gives
$$
|Du_{k,p}|_{L^p(\Omega)}\le \frac{1}{\mu_p^{k-1}}|Du_{1,p}|_{L^p(\Omega)}\le\frac{1}{\mu_p^{k-1+1/p}}|u_{0}|_{L^p(\Omega)}\le
\frac{|\Omega|^{1/p}}{\mu_p^{k-1+1/p}}|u_{0}|_{L^\I(\Omega)}.
$$
Assume $p_0>n$. For $p>p_0$, we can use the above inequality with H\"{o}lder's inequality to get
$$
|Du_{k,p}|_{L^{p_0}(\Omega)}\le |\Omega|^{\frac{1}{p_0}-\frac{1}{p}}|Du_{k,p}|_{L^p(\Omega)}\le
 \frac{|\Omega|^{1/p_0}}{\mu_p^{k-1+1/p}}|u_{0}|_{L^\I(\Omega)}.
$$
By Morrey's inequality and $\lim_{p\rightarrow\infty}\mu_p=\lambda_\infty$, 
$$
(u_{k,p})_{p>p_0}\subset C^{1-n/p_0}(\Omega)
$$
is bounded for each $k\in \N$.  Therefore, the Arzel\`{a}-Ascoli Theorem and a typical diagonalization argument implies there is a sequence 
$(v_k)_{k\in \N}\subset C^{1-n/p_0}(\Omega)$ and a sequence of positive numbers $(p_j)_{j\in \N}$ that are increasing and unbounded such that
$$
v_k=\lim_{j\rightarrow \infty}u_{k,p_j} 
$$
in $C^{1-n/p_0}(\Omega)$ for each $k\in \N$.  

\par Now let $p>r$, and employ H\"{o}lder's inequality and \eqref{ThreeHalfsIneq} to get 
\begin{align*}
\left(\frac{1}{|\Omega|}\int_\Omega|Du_{k,p}|^rdx\right)^{1/r}&\le \left(\frac{1}{|\Omega|}\int_\Omega|Du_{k,p}|^pdx\right)^{1/p}\\
&\le\left(\frac{1}{|\Omega|}\frac{1}{\mu_p}\int_\Omega|u_{k-1,p}|^pdx\right)^{1/p}\\
&\le \frac{1}{\mu_p^{1/p}}|u_{k-1,p}|_{L^\infty(\Omega)}.
\end{align*}
The sequence $(u_{k,p_j})_{j\ge j_r}$ is then bounded in $W^{1,r}_0(\Omega)$ for some $j_r\in \N$ large enough and thus converges to $v_k$ weakly. Therefore, we can substitute $p=p_j$ above and send $j\rightarrow\infty$ to arrive at
$$
\left(\frac{1}{|\Omega|}\int_\Omega|Dv_{k}|^rdx\right)^{1/r}\le |v_{k-1}|_{L^\infty(\Omega)}.
$$
for each $k\in \N$. And after sending $r\rightarrow \infty$,
\begin{equation}\label{MuYImporTante}
|Dv_k|_{L^\infty(\Omega)}\le |v_{k-1}|_{L^\infty(\Omega)}.
\end{equation}
In particular, we have verified that $(v_k)_{k\in\N}\subset W^{1,\infty}_0(\Omega)$.

\par We will now verify that $v_k$ are viscosity solutions of the iteration scheme \eqref{InfinityIterationScheme}.  By induction, 
it suffices to prove this for $k=1$.  Assume $\phi\in C^2(\Omega)$ and $v_1-\phi$ has a local maximum at $x_0\in \Omega$.
We aim to show, 
\begin{equation}\label{ViscSolnCond}
0\ge 
\begin{cases}
\min\{-\Delta_\infty \phi(x_0), |D \phi(x_0)| - u_0(x_0)\}, \quad &u_0(x_0)>0,\\
-\Delta_\infty \phi(x_0), \quad &u_{0}(x_0)=0,\\
\max\{-\Delta_\infty \phi(x_0), -|D\phi(x_0)| - u_{0}(x_0)\}, \quad &u_{0}(x_0)<0.\\
\end{cases}
\end{equation}

\par After adding $x\mapsto \frac{\rho}{2}|x-x_0|^2$ to $\phi$ and later sending $\rho\rightarrow 0^+$, we may assume that $v_1-\phi$ has a {\it strict} local maximum.  Since $u_{1,p_j}$ converges to $v_1$ uniformly on $\Omega$, there is a sequence $(x_j)_{j\in \N}\subset \Omega$ converging to 
$x_0$ for which $u_{1,p_j}-\phi$ has a local maximum at $x_j$. Since $u_{1,p_j}$ is a viscosity solution of \eqref{IterationScheme} with $k=1$ and 
$p=p_j$, 
\begin{equation}\label{FirstThm2Ineq}
-\Delta_{p_j}\phi(x_j)\le |u_0(x_j)|^{p_j-2}u_0(x_j).
\end{equation}
\par If $u_0(x_0)<0$, then $u_0(x_j)<0$ for all $j$ sufficiently large. By \eqref{FirstThm2Ineq}, 
\begin{equation}\label{SecThm2Ineq}
-\Delta_{p_j}\phi(x_j)=|D\phi(x_j)|^{p_j-4}\left\{|D\phi(x_j)|^2\Delta \phi(x_j)+(p_j-2)\Delta_\infty\phi(x_j)\right\}<0,
\end{equation}
and thus $|D\phi(x_j)|\neq 0$ all large enough $j\in \N$.  Canceling the factor of $|D\phi(x_j)|^{p_j-4}$ in \eqref{SecThm2Ineq}, dividing 
by $p_j-2$ and sending $j\rightarrow\infty$ gives  $-\Delta_\infty\phi(x_0)\le 0$.  Likewise, rearranging \eqref{FirstThm2Ineq} leads to
\begin{equation}\label{ThirThm2Ineq}
-\frac{|D\phi(x_j)|^2\Delta \phi(x_j)}{p_j-2}-\Delta_\infty\phi(x_j)\le \frac{1}{p_j-2}\left(\frac{|u_0(x_j)|}{|D\phi(x_j)|}\right)^{p_j-4}u_0(x_j)^3.
\end{equation}
Therefore, it must also be that $-u_0(x_j)\le |D\phi(x_j)|$ for all $j$ large enough. Hence, \eqref{ViscSolnCond} holds when $u_0(x_0)<0$. 

\par Now suppose $u_0(x_0)=0$. If additionally, $|D\phi(x_0)|=0$, then clearly $-\Delta_\infty\phi(x_0)\le 0$. If $|D\phi(x_0)|\neq 0$, 
we can send $j\rightarrow \infty$ in \eqref{ThirThm2Ineq} to again arrive at $-\Delta_\infty\phi(x_0)\le 0$. Thus, 
\eqref{ViscSolnCond} holds when $u_0(x_0)=0$. 

\par Finally, let us assume that $u_0(x_0)>0$, and that $|D\phi(x_0)|-u_0(x_0)>0$. Then $|D\phi(x_j)|-u_0(x_j)>0$ for all $j\in \N$ 
sufficiently large. Passing to the limit in \eqref{ThirThm2Ineq} again gives $-\Delta_\infty\phi(x_0)\le 0$. In conclusion, \eqref{ViscSolnCond} holds in the case $u_0(x_0)>0$, as well. Therefore, we have verified that $v_1$ is a viscosity subsolution of 
\eqref{InfinityIterationScheme}.  An argument that shows $v_1$ is additionally a viscosity supersolution of 
\eqref{InfinityIterationScheme} can be made similarly, so we leave the details to the reader. 
\end{proof}
\begin{proof}[~Proof of Theorem \ref{Thm2} part $(ii)$] In view of \eqref{MuYImporTante}, 
$$
|Dv_k|_{L^\infty(\Omega)}\le |v_{k-1}|_{L^\infty(\Omega)}\le\frac{1}{\lambda_\infty}|Dv_{k-1}|_{L^\infty(\Omega)}.
$$
Therefore, the sequence $(\lambda_\infty^k|Dv_k|_{L^\infty(\Omega)})_{k\in\N}$ is nonincreasing, and the limit 
$$
L:=\lim_{k\rightarrow\infty}\lambda_\infty^k|Dv_k|_{L^\infty(\Omega)}
$$
exists.  The inequality \eqref{MuYImporTante} also implies 
$$
|v_k|_{L^\infty(\Omega)}\le \frac{1}{\lambda_\infty}|Dv_{k}|_{L^\infty(\Omega)}\le\frac{1}{\lambda_\infty}|v_{k-1}|_{L^\infty(\Omega)}.
$$
Consequently, $(\lambda_\infty^k|v_k|_{L^\infty(\Omega)})_{k\in\N}$ is nonincreasing
and the limit 
$$
M:=\lim_{k\rightarrow\infty}\lambda_\infty^k|v_k|_{L^\infty(\Omega)}
$$
exists, as well.  

\par Observe $\lambda_\infty^k|Dv_k|_{L^\infty(\Omega)}\le \lambda_\infty\left(\lambda_\infty^{k-1}|v_{k-1}|_{L^\infty(\Omega)}\right)$ so that
$$
L\le \lambda_\infty M.
$$
Moreover, $\lambda_\infty^k|v_k|_{L^\infty(\Omega)}\le \frac{1}{\lambda_\infty}\lambda_\infty^k|Dv_{k}|_{L^\infty(\Omega)}$, which implies 
$$
\lambda_\infty M\le L.
$$
Thus, $\lambda_\infty M= L$, and when this quantity is nonzero,
$$
\lambda_\infty=\lim_{k\rightarrow\infty}\frac{|Dv_{k}|_{L^\infty(\Omega)}}{|v_{k}|_{L^\infty(\Omega)}}.
$$

\par Finally, note that that the sequence $(w_k)_{k\in\N}:=(\lambda_\infty^{k}v_{k})_{k\in\N}\subset W^{1,\infty}_0(\Omega)$ satisfies the iteration scheme
$$0=
\begin{cases}
\min\{-\Delta_\infty w_k, |Dw_k| - \lambda_\infty w_{k-1}\}, \quad &w_{k-1}>0\\
-\Delta_\infty w_k, \quad &w_{k-1}=0\\
\max\{-\Delta_\infty w_k, -|Dw_k| - \lambda_\infty w_{k-1}\}, \quad &w_{k-1}<0\\
\end{cases} 
$$
in the sense of viscosity solutions. Therefore, if a subsequence of $(\lambda_\infty^{k}v_{k})_{k\in\N}$ converges uniformly on $\Omega$, the stability of viscosity solutions implies that the limit function is necessarily a solution of \eqref{InfinityGroundStates}.
\end{proof}

\appendix



\begin{thebibliography}{}

\bibitem{biezuner} Biezuner, R; Ercole, G; Martins, E. \emph{Computing the first eigenvalue of the p-Laplacian via the inverse power method.}
J. Funct. Anal. 257 (2009), no. 1, 243--270. 

\bibitem{CIL} Crandall, M.; Ishii, H.; Lions, P.-L. \emph{User's guide to viscosity solutions of second order partial differential equations}. Bull. Amer. Math. Soc. (N.S.) 27 (1992), no. 1, 1--67.

\bibitem{DB} DiBenedetto, E. \emph{$C^{1+\alpha}$-local regularity of weak solutions of degenerate elliptic equations}, Nonlinear Anal., 7 (1983), pp. 827--850.

\bibitem{Evans} Evans, L. C. \emph{Partial differential equations}. 2nd edition. Graduate Studies in Mathematics, 19. American Mathematical Society, Providence, RI, 2010.

\bibitem{HyndLindgren} Hynd, R; Lindgren, E. \emph{A doubly nonlinear evolution for ground states of the $p$-Laplacian}. http://arxiv.org/abs/1404.5077. 

\bibitem{JLM} Juutinen, P.; Lindqvist, P; Manfredi, J. {\em{The $\infty$-eigenvalue problem}}, Arch. Ration. Mech. Anal. 148 (1999), no. 2, 89--105.

\bibitem{JLM2} Juutinen, P.; Lindqvist, P; Manfredi, J. \emph{On the equivalence of viscosity solutions and weak solutions for a quasi-linear equation.} SIAM J. Math. Anal. 33 (2001), no. 3, 699--717

\bibitem{LindKaw} Kawohl, Bernd; Lindqvist, Peter. \emph{Positive eigenfunctions for the p-Laplace operator revisited.} Analysis (Munich) 26 (2006), no. 4, 545--550. 

\bibitem{Mazya} Maz'ya, V. \emph{Sobolev Spaces: with Applications to Elliptic Partial Differential Equations}. Second Edition. Springer-Verlag, Berlin, 2011.

\bibitem{MS}  Mikayelyan, H.; Shahgholian H. \emph{Hopf's lemma for a class of singular/degenerate PDE-s}. Annales-Academiae Scientiarum Fennicae Mathematica. 02/2014; 40. 

\bibitem{Saka} Sakaguchi, S. \emph{Concavity properties of solutions to some degenerate quasilinear elliptic Dirichlet problems}.  Ann. Scuola Norm. Sup. Pisa Cl. Sci. (4) 14 (1987), no. 3, 403--421



\end{thebibliography}
\end{document}